\documentclass[12pt]{amsart}
\theoremstyle{plain}
\newtheorem{thm}{Theorem}[section]
\newtheorem{theorem}[thm]{Theorem}

\newtheorem{lemma}[thm]{Lemma}

\newtheorem{proposition}[thm]{Proposition}
\theoremstyle{definition}
\newtheorem{remark}[thm]{Remark}

\newtheorem{notation}[thm]{Notation}
\newtheorem{definition}[thm]{Definition}

\numberwithin{equation}{section}
\newcommand{\sA}{{\mathcal A}}
\newcommand{\sB}{{\mathcal B}}
\newcommand{\sC}{{\mathcal C}}
\newcommand{\sD}{{\mathcal D}}
\newcommand{\sE}{{\mathcal E}}
\newcommand{\sF}{{\mathcal F}}

\newcommand{\sH}{{\mathcal H}}

\newcommand{\sK}{{\mathcal K}}

\newcommand{\sM}{{\mathcal M}}
\newcommand{\sN}{{\mathcal N}}
\newcommand{\sO}{{\mathcal O}}

\newcommand{\sS}{{\mathcal S}}

\newcommand{\sU}{{\mathcal U}}
\newcommand{\sV}{{\mathcal V}}
\newcommand{\sW}{{\mathcal W}}

\newcommand{\sZ}{{\mathcal Z}}

\newcommand{\C}{{\mathbb C}}

\newcommand{\rk}{{\rm rk}}

\title[Webs of Lagrangian tori]{Webs of Lagrangian tori in projective symplectic manifolds}

\author{Jun-Muk Hwang, Richard M. Weiss}

\address{Jun-Muk Hwang, Korea Institute for Advanced Study, Hoegiro 87, Seoul, 130-722, Korea} \email{jmhwang@kias.re.kr}

\address{Richard M. Weiss, Department of Mathematics,
Tufts University, 503 Boston Avenue,  Medford, MA 02155, USA }
\email{rweiss@tufts.edu}

\thanks{Jun-Muk Hwang is supported
by National Researcher Program 2010-0020413 of NRF and MEST, and
Richard M. Weiss is partially supported by DFG Grant KR 1669/7-1}

\begin{document}

\maketitle

\begin{abstract} For a Lagrangian torus $A$ in a simply-connected projective symplectic
 manifold $M$, we prove that $M$ has a
hypersurface disjoint from a deformation of $A$. This implies that
a Lagrangian torus in a compact hyperk\"ahler manifold is a fiber
of an almost holomorphic Lagrangian fibration, giving an
affirmative answer to a question of Beauville's. Our proof employs
two different tools:   the theory of action-angle variables for
algebraically completely integrable Hamiltonian systems  and
Wielandt's theory of subnormal subgroups.
\end{abstract}

\bigskip
\noindent {\sc Keywords.} holomorphic symplectic geometry, compact
hyperk\"ahler manifold, holomorphic Lagrangian torus, action-angle variables, subnormal subgroups

\medskip
 \noindent {\sc AMS Classification.} 14J40, 32G10, 53B99,
20D35

\section{Introduction}

The goal of this paper is to prove the following.
See  Definition \ref{d.symplectic} for the terminology.

\begin{theorem}\label{t.main} Let $M$ be a simply-connected projective
 manifold with a (holomorphic) symplectic form  and let $A \subset M$ be a Lagrangian torus. Then
$M$ has a hypersurface  disjoint from a deformation of $A.$
\end{theorem}

Recall that a simply-connected compact K\"ahler manifold with a (holomorphic)
symplectic form $\omega$  is called a compact hyperk\"ahler
manifold if $H^0(M, \Omega^2_M) = \C \omega$ (cf.\cite{Hu}).
One central problem in compact hyperk\"ahler manifolds is to find
a good condition for the existence of holomorphic or almost holomorphic fibrations on a compact hyperk\"ahler manifold. In the survey \cite{Be} of problems in hyperk\"ahler geometry,
 Beauville asked whether the existence of a Lagrangian torus in $M$ gives rise to such a fibration (Question 6 in \cite{Be}). As observed
 by
Greb-Lehn-Rollenske (Corollary 5.6 of \cite{GLR}),  Theorem
\ref{t.main} implies the following, which gives an affirmative
answer to Beauville's question.

\begin{theorem}\label{t.beauville} Let $A \subset M$ be a Lagrangian torus
in a compact hyperk\"ahler manifold. Then there exists   a
meromorphic map $f: M \dasharrow B$ dominant over a projective
variety $B$, such that on a nonempty Zariski open subset $M^o
\subset M$ with $A \subset M^o$, $f|_{M^o}$ is a proper smooth
morphism and $A$ is a fiber of $f$.
\end{theorem}

 The deduction of
Theorem \ref{t.beauville} from Theorem \ref{t.main} is a
combination of a number of prominent results in  hyperk\"ahler
geometry, in particular, \cite{COP} and \cite{Vo}, as well as
the standard hyperk\"ahler machinery (\cite{Hu}). We will not discuss
this deduction, referring the reader to \cite{GLR}.

%

  Our
proof of Theorem \ref{t.main} uses completely different ideas and
requires little knowledge of hyperk\"ahler geometry.
There are two crucial ingredients in our proof, one geometric and one algebraic. It is easy to see
that deformations of a Lagrangian torus $A \subset M$ give rise to
a multi-valued holomorphic foliation on a Zariski open subset in
$M$. If this foliation is univalent, Theorem \ref{t.main}  is
easily obtainable. Thus the key issue  is how to deal with the
multi-valuedness. To handle this difficulty, we are going to study
the monodromy action of this multi-valued foliation (cf.
Definition \ref{n.group}). The main geometric ingredient,
Proposition \ref{p.integrable}, of our proof is the integrability
of the local distribution given by a pair of sheets of the
multi-valued foliation. This is established by means of the theory
of action-angle variables for completely integrable Hamiltonian
systems (see, e.g., \cite{GS}, Section 44). This `pairwise integrability'
gives  some restrictions on the monodromy action, which is an
action of a finite group on a finite set. However, these
restrictions on the monodromy action do not immediately give us a
solution of the problem. It turns out that a non-trivial result on
the actions of finite groups on finite sets is required. This is
our key algebraic ingredient, Theorem \ref{t.triple}. Logically
speaking, it belongs to abstract group theory, independent of
geometry. Its proof uses Wielandt's work on subnormal subgroups
(\cite{Wi}) and may be of independent interest.

%


\bigskip
\noindent {\bf Acknowledgment} $\;$   When we first started working on Beauville's question (Theorem \ref{t.beauville}),
it was Keiji Oguiso who told us that it could be reduced to proving Theorem \ref{t.main}. We would like to
thank him for this information and encouragement.
%

\bigskip
\noindent {\bf Conventions}

\noindent 1.
Throughout the paper, a manifold is always connected, unless
stated otherwise. A variety may have finitely many irreducible
components. A projective manifold is a nonsingular irreducible
projective variety.

\noindent 2.  When we say an open set, we mean it  in the
classical topology. An open set in Zariski topology will be called
a Zariski open set.

\noindent 3. For a  projective subvariety $A$ in an algebraic
variety $M$, $[A] \in {\rm Hilb}(M)$ denotes the point of the Hilbert scheme determined by $A$. We abuse the
 term `deformation' as follows. A {\em deformation} of $A$ means a subvariety of $M$ corresponding to a point in
an irreducible Zariski open subset containing $[A]$ in ${\rm Hilb}(M)$. A {\em small deformation} of $A$ means a subvariety of $M$
corresponding to a point in a classical neighborhood (or the germ) of $[A]$ in ${\rm Hilb}(M)$.

\section{A result on actions of finite groups on finite sets}

This section is devoted to the study of certain actions of finite
groups.   The content of this section will  be used only at the
very end of Section \ref{s.Lagrangian}, where we use  Theorem
\ref{t.triple} to complete the proof of Theorem \ref{t.main}.

\begin{definition}\label{d.subnormal}
A subgroup $H$ of a group $G$ is {\em subnormal in } $G$ if there
exist a natural number $\ell$ and a chain of subgroups $$H =
F_{\ell} \subset F_{\ell -1} \subset \cdots \subset F_1 \subset
F_0 = G$$ such that $F_i$ is normal in $F_{i-1}$ for all $i=1,2,
\ldots, \ell.$ \end{definition}

Our main tool is the following result of Wielandt proved in Satz 2
of \cite{Wi}. See 6.7.4 of \cite{KS}  for a particularly beautiful
proof.

\begin{theorem}\label{t.Wielandt}
Let $G$ be a finite group. For a subgroup $H \subset G$ and $g \in
G$, denote by $\langle H, gHg^{-1} \rangle$ the subgroup generated
by $H$ and $gHg^{-1}$. If $H$ is subnormal in $\langle H, gHg^{-1}
\rangle$ for all $g \in G$, then $H$ is subnormal in $G$.
\end{theorem}

\begin{definition}\label{d.triple}
We will consider  triples $(X,G,H)$ consisting of a finite set
$X$, a finite group $G$ acting on $X$ transitively, and a normal
subgroup $H \lhd G_x$ of the stabilizer $G_x$ of a distinguished
point $x \in X$. Given such a triple and an element $y \in X$,
define $H_y:=gHg^{-1},$ where $g \in G$ is an element such that
$y=g \cdot x$. Since $H$ is a normal subgroup of $G_x$, $H_y$ is
independent of the choice of
 $g$. Given a subset $Y \subset X$,   say,
  $Y = \{ y_1, \ldots, y_m\}$, denote by $|Y|$ the cardinality $m$ of $Y$ and denote by $\langle Y \rangle$, or
   $\langle y_1, \ldots, y_m \rangle$, the subgroup of $G$
 generated by $\cup_{y \in Y}H_y$. For example, $H= \langle x \rangle$.
 A triple $(X,G,H)$ will be called {\em trivial} if $|X|=1$.
  A triple $(X,G,H)$ will be called {\em special} if the following
two conditions are satisfied.  \begin{enumerate} \item[(1)]
$\langle X \rangle$ acts transitively on $X$. \item[(2)] For any
two distinct elements $y\neq z \in X$, $y$ and $z$ are not in the
same $\langle y,z \rangle$-orbit.
\end{enumerate}
\end{definition}

Our result is the following.

\begin{theorem}\label{t.triple}
There are no non-trivial special triples.
\end{theorem}

\begin{proof} Suppose that  there exists a special triple $(X,G,H)$
  with $|X| >1$. Choose one such $(X,G,H)$ with minimal possible $|X|>1$
and among those with minimal $|X|$, one with minimal $|G|$. If
$(X, G, H)$ is a special triple, then so is $(X, \langle X
\rangle, H)$. By the minimality of $|G|$, we have $G= \langle X
\rangle$.

     \begin{lemma}\label{l.1}
     There is no normal subgroup $N \lhd G$ such that $H \subset N \neq G$. In particular, $H$ is not
     a subnormal subgroup of $G$. \end{lemma}

     \begin{proof}
     Assume the contrary and choose such a normal subgroup $N$. Then for any $g \in G$,
     $$gHg^{-1} \subset gNg^{-1} = N.$$ Thus $H_y \subset N \neq G$ for all $y \in X$. This contradicts $\langle X \rangle =G$. \end{proof}

     \begin{lemma}\label{l.2}
     For a subgroup $F \subset G$, let $F\cdot x$ denote
the $F$-orbit containing $x$ and let $F^{\circ}=\langle F\cdot
x\rangle$. Suppose that $F$ is a proper subgroup of $G$ containing
$H$. Then $H \subset F^{\circ} \lhd F$ and either $F^{\circ} \neq
F$ or $H=F^{\circ}=F$.
  \end{lemma}

  \begin{proof} By definition, $F^{\circ}$ is generated by $\{ fHf^{-1}, f \in F\}.$
  Since $H \subset F$, we have $H \subset F^{\circ} \lhd F$.

  Suppose that  $F^{\circ}=F$. Then $F^{\circ}$ acts transitively on $F \cdot x$ and $F=\langle F\cdot x \rangle$.
 This implies that $(F \cdot x, F, H)$ is a special triple.
  If $|F \cdot x| \neq 1$, then, by the minimality assumption,
   $F \cdot x = X$ implying $$F=F^{\circ} =   \langle F \cdot x\rangle = \langle X \rangle = G,$$ a contradiction.
  Thus $|F \cdot x| = 1$, which implies  $$F=F^{\circ}= \langle F\cdot x\rangle=\langle x \rangle=H.$$ \end{proof}

 Now we derive a contradiction as follows. Pick any $y \neq x \in X$.
 Define $F_1 :=\langle x,y \rangle$. By Definition \ref{d.triple} (2),
 $x$ and $y$ are in two different $F_1$-orbits. This implies
 that $H \subset F_1 \neq G$. If $H=F_1$, we stop here. If $H \neq F_1$,
 an application of Lemma \ref{l.2},
   gives $F_2:= F_1^{\circ}$ satisfying $$H \subset F_2 \lhd F_1$$ with $F_2 \neq F_1$.
    If $H=F_2$, we stop. Otherwise, we can repeat
   the process to get $F_3=F_2^{\circ}$ satisfying $$H \subset F_3 \lhd F_2 \lhd F_1$$
    with $F_3 \neq F_2$. Repeating this,
   we get a natural number $\ell$ and a sequence of subgroups
   $$H = F_{\ell} \lhd F_{\ell-1} \lhd \cdots \lhd F_2 \lhd F_1$$
   such that $F_{i}$ is a proper normal subgroup of $F_{i-1}$ for each $i$. Thus $H$ is subnormal in
   $F_1= \langle x,y \rangle$ for any choice of $y$. In other words, $H$ is subnormal in
   $\langle H, gHg^{-1} \rangle$ for any $g \in G$. By Wielandt's Theorem,
   $H$ is subnormal in $G$. This is a contradiction to Lemma \ref{l.1}.
   \end{proof}

\section{Webs of submanifolds}

As explained in the introduction, our main object of study is a multi-valued foliation on a projective manifold
arising from deformations of an algebraic submanifold. It is convenient to introduce the following to
describe such a multi-valued foliation.

\begin{definition}\label{d.web} Let $M$ be a projective manifold.
A {\em web of submanifolds} on $M$  is the following data, to be
denoted by $\sW= [\mu: \sU \to M, \rho: \sU \to \sK]$.
\begin{itemize} \item[(1)] A generically finite surjective
morphism $\mu: \sU \to M$ from a projective manifold $\sU$.
\item[(2)] A projective morphism $\rho: \sU \to \sK$ with
connected fibers onto a projective manifold $\sK$ with a Zariski
open subset $\sK^{\rm bihol} \subset \sK$ such that for each $a
\in \sK^{\rm bihol}$,
\begin{enumerate} \item[(i)] $\rho^{-1}(a)$ is smooth; \item[(ii)]
$\mu|_{\rho^{-1}(O_a)}: \rho^{-1}(O_a) \to \mu(\rho^{-1}(O_a))$ is
biholomorphic for some open neighborhood $O_{a}$ of $a$ in
$\sK^{\rm bihol}$; \item[(iii)] $\mu(\rho^{-1}(b)) \neq
\mu(\rho^{-1}(a))$ if  $b \in \sK, b \neq a.$ \end{enumerate}
\end{itemize}
 For a point $a \in \sK^{\rm bihol}$, the submanifold
$\mu(\rho^{-1}(a))$ in $Z$ is called a {\em member of the web}
$\sW$.
\end{definition}

We skip the proof of the following  easy proposition.

\begin{proposition}\label{p.etale}
Given a web $\sW= [\mu: \sU \to M, \rho: \sU \to \sK]$ of
submanifolds, there exists a nonempty Zariski open subset $M^{\rm
et} \subset M$ such that  $$\sU^{\rm et} := \mu^{-1}(M^{\rm et})
\subset \rho^{-1}(\sK^{\rm bihol})$$ and $\mu|_{\sU^{\rm et}}: \sU^{\rm et} \to
M^{\rm et}$ is \'etale.
\end{proposition}

\begin{notation}\label{n.A}
In the setting of Proposition \ref{p.etale}, let $d$ be the degree
of $\mu$. For  $y \in M^{\rm et}$, write $\mu^{-1}(y) = \{ y_1,
\ldots, y_d\}$ and $I = \{ 1, \ldots, d\}$. For each $i \in I$, we
set $$A^i_y := \mu(\rho^{-1}(\rho(y_i))).$$ For each pair $(i,j) \in I \times I$,
let $A^i_{y_j}$ be the irreducible component of $\mu^{-1}(A^i_y)$ containing $y_j$. This
notation is not very precise, because it involves an ordering of
$\mu^{-1}(y)$. But this should not cause  confusion, because it
will be applied to a given point $x \in M^{\rm et}$ and points $y$
in a sufficiently small neighborhood of $x$ where we can always
fix an ordering of $\mu^{-1}(y)$ in a uniform manner.
\end{notation}

We recall the following standard topological fact.

\begin{lemma}\label{l.Ehresman}  Let $\rho: \sU \to \sK$ be a  proper morphism
between projective manifolds with connected fibers. Let $\sE
\subset  \sU$ be a proper subvariety. Then there exists a proper
subvariety $\sB \subset \sK$ such that the restriction $\rho'$ of
$\rho$ to the complement of $\sB$ and $\sE$, i.e.,
$$\rho':  \rho^{-1}(\sK \setminus \sB)  \setminus \sE  \to \sK \setminus \sB,$$ is locally
differentiably trivial over the base in the sense of \cite{CMP}
Theorem 4.1.2. \end{lemma}

\begin{proof}
By Ehresman's Theorem (\cite{CMP} Theorem 4.1.2) it suffices to
prove that $\rho'$ is evenly submersive in the sense of
\cite{CMP}, p. 133. Replacing $\sU$ by a log-resolution of $(\sU,
\sE)$, we may choose a Zariski open $\sK_1 \subset \sK$ such that
$\sE_1:=\rho^{-1}(\sK_1) \cap \sE$ is the union of smooth
hypersurfaces with simple normal crossing and the restriction of
$\rho$ on each component of $\sE_1$ and each intersection stratum
of the components is a smooth morphism. In other words, $\sE_1
\subset \rho^{-1}(\sK_1)$ is  relatively simple normal crossing
with respect to $\rho$. In particular, for a given point $x \in
\sK_1$ and $y \in \rho^{-1}(x)$, there exists a neighborhood $O_x$
of $x$ in $\sK_1$ and a neighborhood $U_y$ of $y$ in
$\rho^{-1}(O_x)$ such that
\begin{itemize} \item[1.] there exists a biholomorphic map $h:
U_y \to W \times O_x$, where $W$ is a domain in $ \C^n, n= \dim
\sU - \dim \sK$, \item[2.] $\rho'|_{U_y} = p_2 \circ h,$ where
$p_2: W \times O_x \to O_x$ is the projection to the second
factor, and \item[3.] $h(\sE \cap U_y)$ is the product of a simple
normal crossing hypersurface in $W$ and $\sO_x.$ \end{itemize}
From the compactness of $\sE \cap \rho^{-1}(x)$, finitely many of
such neighborhoods $U_y$ cover $\sE$. Thus we can fix the
neighborhood $O_x$ for all $y \in \rho^{-1}(x)$. Now if we choose $\sB$
to be the complement $\sK \setminus \sK_1$, then the corresponding
$\rho'$ is evenly submersive.
\end{proof}

\begin{proposition}\label{p.sC}
In the setting of Proposition \ref{p.etale}, there exists a proper
subvariety $\sC \subset \sK$ containing $\sK \setminus \sK^{\rm
bihol}$  such that the restriction of $\rho$ to $\sU^{\rm et}
\setminus \rho^{-1}(\sC)$ is locally differentiably trivial over
$\sK \setminus \sC$ and  the morphism $\mu$ gives an embedding of
$\rho^{-1}(a) \cap \sU^{\rm et}$ into $M^{\rm et}$ for each $a \in
\sK \setminus \sC$.
\end{proposition}

\begin{proof}
Let $\sE \subset \sU$ be $\mu^{-1}(M \setminus M^{\rm et})$ and
apply Lemma \ref{l.Ehresman} to obtain $\sB \subset \sK$.  Setting
$\sC = \sB \cup (\sK \setminus \sK^{\rm bihol})$, we have the
result. \end{proof}

\begin{definition}\label{n.group}
In the setting of Proposition \ref{p.sC}, define $M_o := M^{\rm
et} \setminus \mu(\rho^{-1}(\sC))$. Fix a  point $x \in M_o.$  Let
$X$ be the finite set $\mu^{-1}(x)= \{ x_1, \ldots, x_d\}$ using
Notation \ref{n.A},  and let $\mathfrak{S}_X$ be the symmetry group
on $X$. The \'etale cover $\sU^{\rm et} \to M^{\rm et}$
induces a natural homomorphism
$$ \alpha: \pi_1(M^{\rm et}, x) \to \mathfrak{S}_X$$ whose image will be denoted
by $G$. By the connectedness of $\sU^{\rm et}$, $G$ acts
transitively on $X$. For each $1 \leq i \leq d$, let $H_i \subset
G$ be the image of the homomorphism
$$ \alpha \circ \lambda_i: \pi_1(A^i_x \cap M^{\rm et}, x) \to \mathfrak{S}_X,$$ where
 $$ \lambda_i:  \pi_1(A^i_x \cap M^{\rm et}, x)
\to \pi_1(M^{\rm et}, x)$$ is  induced by the inclusion $A^i_x
\subset M$.
\end{definition}

\begin{proposition}\label{p.conjugate}
In Definition \ref{n.group}, for all $1 \leq i \leq d$, denote by
$G_i \subset G$  the isotropy subgroup of $x_i \in X$. We have
$G_i=g_i G_1 g_i^{-1}$
and $H_i=g_i H_1 g_i^{-1}$ for any $g_i\in G$ with $g_i\cdot x_1=x_i$ and
$H_i$ is a normal subgroup of $G_i$.
\end{proposition}

\begin{proof}
Choose $g'_i \in \pi_1(M^{\rm et}, x)$ with $\alpha (g'_i) = g_i$.
Let
$$\gamma_i: [0,1] \to \sU^{\rm et}$$ be a path representing  $g'_i$,
i.e., $$\gamma_i(0)=x_1, \; \gamma_i(1)= x_i \mbox{ and the class
of } \mu \circ \gamma_i \mbox{ belongs to  } g'_i.$$  Since $X
\cap \rho^{-1}(\sC) = \emptyset$,  we can assume that
$\gamma_i([0,1])$ is disjoint from $\rho^{-1}(\sC)$. Set $c_t :=
\rho(\gamma_i(t)) \in \sK \setminus \sC$. Then the family $$\{\mu
(\rho^{-1}(c_t)) \cap M^{\rm et}, t \in [0,1]\}$$ is locally
differentiably trivial with
$$\mu (\rho^{-1}(c_0)) \cap M^{\rm et} = A^1_x \cap M^{\rm et} \mbox{ and }
\mu(\rho^{-1}(c_1)) \cap M^{\rm et} = A^i_x \cap M^{\rm et}.$$
Therefore, given a closed path $$\beta_0: [0,1] \to A^1_x
\cap M^{\rm et}, \; \beta_0(0) = \beta_0(1) = x$$ representing an element of
$$\pi_1(\mu (\rho^{-1}(c_0))\cap M^{\rm et}, x) = \pi_1(A^1_x \cap M^{\rm et},x),$$ we can find a
continuous family of closed paths
$$\big\{\beta_t: [0,1] \to \mu (\rho^{-1}(c_t))\cap M^{\rm et}, \;
\beta_t(0) = \beta_t(1) = \mu(\gamma_i(t)), \; t \in [0,1] \big\}.$$
Thus in $M^{\rm et}$,  $\beta_0$ is homotopic to
$$(\mu \circ \gamma_i)^{-1} \cdot \beta_1 \cdot (\mu \circ
\gamma_i)$$ for some closed path $\beta_1$ based at $x$
representing an element of $ \pi_1(A^i_x \cap M^{\rm et}, x)$.
This implies that $$[\beta_0] = (g'_i)^{-1} \cdot [\beta_1] \cdot
g'_i \; \mbox{ in } \pi_1(M^{\rm et}, x),$$ proving $H_i= g_i H_1
g_i^{-1}$. Setting $i=1$, we conclude that
$H_1$ is normal in $G_1$. Therefore $H_i=g_i H_1 g_i^{-1}$ is normal in $g_i G_1 g_i^{-1}=G_i$
for all $i.$
\end{proof}

\begin{definition}\label{d.split}
Let $f: M' \to M$ be a generically finite surjective morphism
between two irreducible nonsingular varieties.  Given an
irreducible subvariety $A \subset M$, we say that $f$ {\em splits
over} $A$ if for each irreducible component $A'$ of $f^{-1}(A)$
satisfying $f(A') = A$, the restriction $f|_{A'}: A' \to A$ is
birational.
\end{definition}

\begin{proposition}\label{p.equivalence}
In Definition \ref{n.group}, denote by $\sH \subset G$ the
subgroup  generated by $H_1, \ldots, H_d.$ Assume that $\sH$ does
not act transitively on $X$. Then there exists a projective
manifold $M'$ and a generically finite surjective morphism $f: M'
\to M$ which is not birational and splits over $\mu(\rho^{-1}(a))$
for a general $a \in \sK$.
\end{proposition}

\begin{proof} Put $\sU_o :=
\mu^{-1}(M_o),$ where $M_o$ is as in Definition \ref{n.group}.
Given two point $u, v \in \sU_o$ with $\mu(u) = \mu(v)$, write $u
\sim v$ if the following holds: there exist a point $w \in \sU_o$
with $\mu(w) = \mu(u)=\mu(v)$ and an irreducible component of
$\mu^{-1}(\mu(\rho^{-1}(\rho(w)))\cap M^{\rm et})$ containing both
$u$ and $v$. Now let $\approx$ be the equivalence relation on
$\sU_o$ generated by $\sim$. In other words, two points $u$ and
$v$ are equivalent, $u \approx v$, if there exists an integer
$\ell \geq 1$  and a sequence of points in $\sU_o$
$$u=u_1, u_2, \ldots, u_{\ell-1}, u_{\ell} = v $$
such that $u_i \sim u_{i+1} $ for each $1 \leq i \leq \ell -1$.
Over a Zariski open subset $M_1 \subset M_o$, this gives an
\'etale equivalence relation, i.e., the equivalence classes on $\sU_1:=
\mu^{-1}(M_1)$ determine an \'etale factorization $\sU_1 \to M'_1
\to M_1$ of $\mu|_{\sU_1}$. From the definition of the equivalence
relation, the \'etale morphism $M'_1 \to M_1$ splits over
$\mu(\rho^{-1}(a)) \cap M_o$ for a general $a \in \sK$.

We claim that $M'_1 \to M_1$ is not bijective  if $\sH$ does not
act transitively on $X= \mu^{-1}(x), x \in M_1$. To see the claim,
note that for two point $x_i,x_j \in X$, $x_i \sim x_j$  if and
only if $A^k_{x_i} = A^k_{x_j}$ for some $k \in I=\{ 1, \ldots,
d\}$, using Notation \ref{n.A}. The latter is equivalent to  $x_j
= H_k \cdot x_i$ for some $k$. Thus the equivalence classes of
$\approx$ in $X$ are just the orbits of the group $\sH$. This proves
the claim.

Since the equivalence relation $\approx$ on $\sU_1$ is an
algebraic equivalence relation, we have a generically
finite morphism $f:M' \to M$ compactifying $M'_1 \to M_1$ and a
dominant rational map $q: \sU \dashrightarrow M'$ satisfying $\mu
= f \circ q$. By the assumption that the group $\sH$ does not act
transitively on $X$ and the previous claim, we see that $f$ is not
birational. By the definition of $M'$, we know that $f$ splits
over $\mu(\rho^{-1}(a))$ for a general $a \in \sK$.
\end{proof}

\begin{proposition}\label{p.split}
Let $\sW= (\mu: \sU \to M, \rho: \sU \to \sK)$ be a web of
submanifolds on a projective manifold $M$.   Let $f: M' \to M$ be
a generically finite surjective morphism.  Assume that $f$ splits
over $\mu(\rho^{-1}(a))$ for a general $a \in \sK$. Then
 $\mu(\rho^{-1}(a))$ is disjoint from the reduced branch divisor $D \subset M$
 of the morphism $f$.
\end{proposition}

\begin{proof}
Suppose that $\mu(\rho^{-1}(a))$ has non-empty intersection with
$D$. By the generality of $a \in \sK$, we can assume that
$\mu(\rho^{-1}(a))$ passes through a general point $y$ of an
irreducible component of $D$ and the divisor $D':= D \cap
\mu(\rho^{-1}(a))$ on $\mu(\rho^{-1}(a))$ is  smooth at $y$. Pick
a ramification point $z \in M'$ with $f(z) =y$. Then near $z$, the
morphism is locally analytically equivalent to a cyclic branched
covering of degree $\geq 2$. Pick an irreducible curve $C \subset
\mu(\rho^{-1}(a))$ such that $C$ intersects $D'$ transversally at
$y$. Then $f^{-1}(C)$ has an irreducible component $C'$ through
$z$ such that $C' \to C$ is locally a cyclic branched covering near
$z$ of degree $\geq 2$. The irreducible component of
$f^{-1}(\mu(\rho^{-1}(a)))$ containing $z$ cannot be birational
over $\mu(\rho^{-1}(a))$ because it must contain $C'$ and we can
choose $C$ to pass through any general point of
$\mu(\rho^{-1}(a))$. This contradicts the assumption that $f$
splits over $\mu(\rho^{-1}(a))$. \end{proof}

\section{Pairwise integrable webs of submanifolds}

The term `web' in the previous section has its origin in `web
geometry' in  differential geometry. In this section, we need to
view a web from this original viewpoint of local differential
geometry.  To be precise,  we  introduce the following definition.

\begin{definition}\label{d.regular} Let $U$ be a complex manifold.
A {\em regular web} on $U$ is a finite number of integrable
subbundles $$W^i \subset T(U), \; i \in I := \{1, 2, \ldots, d\}$$
for some integer $d \geq 1$ such that for any pair $(i,j) \in   I\times I$, the intersection
$W^i\cap W^j \subset T(U)$ is also a subbundle. This implies that
the sum $W^{ij} = W^{ji} \subset T(U)$ with fiber at $x \in U$
$$W^{ij}_x := W^i_x + W^j_x \subset T_x(U)$$ is a subbundle of
rank $\rk(W^i) + \rk(W^j) - \rk(W^i\cap W^j)$. A regular web is
{\em pairwise integrable} if for any pair $(i,j)$, $W^{ij} =
W^{ji}$  is integrable. \end{definition}

\begin{remark} We are interested in local differential geometry of a regular web. So we will
 assume that all leaves of integrable
distributions are closed in the complex manifold $U$.
\end{remark}

The following three lemmata are immediate.

\begin{lemma}\label{l.generic} Let $U$ be a complex manifold and let  $\{ W^i \subset T(U), i \in
I\}$ be an arbitrary finite collection of integrable subbundles of
$T(U)$. Then there exists a nonempty Zariski open subset
$U'\subset U$ such that the restriction $\{W^i|_{U'}, i \in I\}$
defines a regular web on $U$. \end{lemma}

\begin{lemma}\label{l.regular} Let $\{W^i, i \in I\}$ be a regular
web on a complex manifold $U$. Given a connected open subset $U'
\subset U$, the restriction $\{W^i|_{U'}, i \in I \}$
is a regular web on $U'$. If $\{W^i|_{U'}, i \in I \}$ is pairwise integrable
for some nonempty $U' \subset U$, then $\{W^i, i \in I\}$ is also pairwise
integrable.
\end{lemma}

\begin{lemma}\label{l.leaf} Given a regular web $\{W^i, i \in I\}$
 on a complex manifold $U$ and a point $x\in U$,
denote by $L^i_x$ the leaf of $W^i$ through $x$. Then $L^i_y = L^i_x$ for any $y \in L^i_x$.
 When $W^{ij}$ is integrable, denote by $L^{ij}_x$ the leaf of $W^{ij}$ through $x$.
 Then $L^i_x, L^j_x \subset L^{ij}_x$ and for any $y \in L^{ij}_x$,  $$L^{ij}_x = L^{ji}_x = L^{ji}_y = L^{ij}_y.$$ \end{lemma}

\begin{proposition}\label{p.local}
When $W^{ij}$ is integrable in Lemma \ref{l.leaf},  the
germ of $L^{ij}_y$ at $y \in U$ is equal to that of $\cup_{z \in L^i_y}
L^j_z$ and also that of $\cup_{z \in L^j_y} L^i_z.$
\end{proposition}

\begin{proof}
 Set $\rk(W^i) = r$
and $\dim U = n$. Let $\Delta^r \times \Delta^{n-r}$ be the
product of polydiscs of dimension $r$ and $n-r$, respectively.  Let $p:
\Delta^r \times \Delta^{n-r} \to \Delta^{n-r}$ be the projection.
There exists a neighborhood $U_y$ of $y$ biholomorphic to
$\Delta^r \times \Delta^{n-r}$ such that $L^i_x = p^{-1}(p(x))$
when we identify $U_y$ with the product of polydiscs. By the
requirement that $W^{ij}$ is a vector subbundle of $T(U)$,
$p|_{L^j_y \cap U_y}$ is a smooth morphism and by shrinking $U_y$
if necessary, we can assume that $p(L^j_y \cap U_y)$ is a
submanifold of dimension $\rk(W^j) - \rk(W^i \cap W^j)$ in
$\Delta^{n-r}$. The germ of $\cup_{z \in L^i_y} L^j_z$ at $y$ is
equal to that of the submanifold $p^{-1}(p(L^j_y \cap U_y))$ and has dimension
 $$ r + \dim (p(L^j_y \cap U_y)) =  r + \rk(W^j) - \rk(W^i \cap W^j)= \dim (L_y^{ij}).$$
Because
 $ L^i_y \subset L^{ij}_y$ and
for each $z \in L^i_y$, $L^j_z \subset L^{ij}_z = L^{ij}_y$ by
Lemma \ref{l.leaf}, we have
$$\bigcup_{z \in L^i_y} L^j_z \subset L^{ij}_y.$$
 Since the germs of both sides at $y$ are
smooth and both sides have the same dimension, their germs at $y$
coincide.
\end{proof}

\begin{definition}\label{d.good}
Given a web $\sW= [\mu: \sU \to M, \rho: \sU \to \sK]$ of
submanifolds on a projective manifold, let $M_o\subset M$ be the
Zariski open subset in Definition \ref{n.group} and let $\sU_o :=
\mu^{-1}(M_o)$. Denote by $d \geq 1$  the degree of $\mu$. For
each point $x \in M_o,$ there exists a neighborhood $U \subset
M_o$ of $x$ such that
$$\rho^{-1}(U) = U^1 \cup U^2 \cup \cdots \cup U^d$$ is a disjoint
union of connected open subsets $U^i \subset \sU$ and each $\mu^i:= \mu|_{U^i}$ is a
biholomorphic morphism from $U^i$ onto $U$. Let $T^{\rho}(U^i)
\subset T(U^i)$ be the vector subbundle given by the relative
tangent bundle of the smooth morphism $\rho|_{\sU_o}$.  Let $W^i=
d \mu^i (T^{\rho}(U^i))$ be the corresponding vector subbundle of
$T(U)$. We say that $x \in M_o$ is {\em a good point} with respect
to the web $\sW$ if $\{ W^1, \ldots, W^d \}$  defines a regular
web (in the sense of Definition \ref{d.regular}) in a neighborhood
of $x$. Let $M^{\rm good} \subset M_o$ be the subset
consisting of good points. By Lemma \ref{l.generic}, $M^{\rm
good}$ is a nonempty Zariski open subset in $M$. For each point $x
\in M^{\rm good}$, the regular web $\{W^1, \ldots, W^d\}$ in
a neighborhood of $x$ is called the {\em induced regular web} (at
$x$). The web $\sW$ of submanifolds is {\em pairwise integrable}
if the induced regular web at some (hence any by Lemma
\ref{l.regular}) point $x\in M^{\rm good}$ is pairwise integrable
in the sense of Definition \ref{d.regular}.
\end{definition}

\begin{notation}\label{n.pair}
In the setting of Definition \ref{d.good}, for each pair $(i,j)
\in I \times I$, recall from Notation \ref{n.A} that  $A^i_{y_j} \subset \sU$ is  the
irreducible component of $\mu^{-1}(A^i_y)$ passing through $y_j$. In particular, $A^i_{y_i} = \rho^{-1}(\rho(y_i))$.
We set $$K^i_{y_j} := \rho(A^i_{y_j}), \; A^{ij}_{y_j}:= \mbox{
closure of } \rho^{-1}(K^i_{y_j} \cap \sK^{\rm bihol})  \mbox{ and
} A^{ij}_y := \mu(A^{ij}_{y_j}).$$ All of these are irreducible subvarieties.  \end{notation}

\begin{proposition}\label{p.A^ij}
In the setting of Definition \ref{d.good}, choose a neighborhood
$x \in U \subset M^{\rm good}$ and let $\{W^1,\ldots,W^d\}$ be the regular web on $U$ obtained from $T^\rho(U^i).$
Using notation of Lemma \ref{l.leaf} and Notation
\ref{n.pair}, we have the
following for any $y \in U$.
\begin{enumerate} \item[(i)]  $L^i_y$ is the connected component of $A^i_y \cap U$ through
$y$.
\item[(ii)] If $\sW$ is pairwise integrable, then $L^{ij}_y$ is an irreducible component
of $A^{ij}_y \cap U$.\end{enumerate}\end{proposition}

    \begin{proof} (i) is immediate from the definition of $W^i$.
    (ii) is a consequence of (i) and Proposition \ref{p.local}.
\end{proof}

\begin{proposition}\label{p.leaf}
Let $\sW=[\mu: \sU \to M, \rho: \sU \to \sK]$ be a pairwise
integrable web of submanifolds on a projective manifold $M$. For a
point $x\in M^{\rm good}$, choose a neighborhood $x \in U \subset
M^{\rm good}$ equipped with the induced regular web $\{W^1, \ldots,
W^d\}$ on $U$. Using Notation \ref{n.pair}, we have the following.
\begin{enumerate}
    \item[(i)] $A^i_y, A^j_y \subset A^{ij}_y$, $\rho(A^{ij}_{y_j}) = \rho(A^i_{y_j})$ and $A^{ij}_{y_j}$
    is an irreducible component of $\mu^{-1}(A^{ij}_y)$ through $y_j$.
        \item[(ii)]  $A^{ij}_y = A^{ji}_y.$
\item[(iii)] If $y \in L^{ij}_x$, then $A^{ij}_y= A^{ij}_x$
 \item[(iv)] If $y \in L^{ij}_x$, $ A^{ij}_{y_i} = A^{ij}_{x_i}$ and  $A^{ij}_{y_j} = A^{ij}_{x_j}.$
\end{enumerate}
\end{proposition}

    \begin{proof}
    (i) is immediate from the definition in Notation \ref{n.pair}.
    (ii) and (iii) follow from Lemma \ref{l.leaf} and Proposition \ref{p.A^ij}(ii).
    (iv) follows from (i) and (iii). \end{proof}

    \begin{proposition}\label{p.sigma}
In the setting of Proposition \ref{p.leaf}, let $\widehat{A^{ij}_x}$
be the normalization of $A^{ij}_x$ and let
$$\sigma: \widetilde{A^{ij}_x} \to \widehat{A^{ij}_x} \to A^{ij}_x \subset M$$ be a
desingularization of $A_x^{ij}$ which leaves the smooth locus of
$\widehat{A^{ij}_x}$ intact. For any $y \in L^{ij}_x$, let  $\widetilde{A^i_y}$ and $ \widetilde{A^j_y} \subset
\widetilde{A^{ij}_x}$ be the proper transforms of $A^i_y$ and $A^j_y$, respectively.
 Then there exist   neighborhoods $\sV^i_y$  of  $\widetilde{A^i_y}$ and  $\sV^j_y$ of
 $\widetilde{A^j_y}$ in $\widetilde{A^{ij}_x}$ equipped
 with  morphisms $$\upsilon^i_y: \sV^i_y \to A^{ij}_{y_i} \cap \rho^{-1}(\sK^{\rm bihol})
  \mbox{ and } \upsilon^j_y: \sV^j_y \to A^{ij}_{y_j}\cap \rho^{-1}(\sK^{\rm bihol})$$
  such that \begin{enumerate} \item[(1)]  $\upsilon^i_y(\widetilde{A^i_y}) = A^i_{y_i}$ and
$\upsilon^j_y(\widetilde{A^j_y}) = A^j_{y_j};$ \item[(2)] the four
morphisms $$\sigma|_{\sV^i_y}: \sV^i_y \to M, \;
\sigma|_{\sV^j_y}: \sV^j_y \to M, \; \upsilon^i_y: \sV^i_y \to \sU
\mbox{ and } \upsilon^j_y: \sV^j_y \to \sU$$ are embeddings;
 \item[(3)] the four morphisms $$\sigma|_{\widetilde{A^i_y}}: \widetilde{A^i_y} \to
 A^i_y,\;
 \sigma|_{\widetilde{A^j_y}}: \widetilde{A^j_y} \to A^j_y, \; \upsilon^i_y|_{\widetilde{A^i_y}}:
 \widetilde{A^i_y} \to A^i_{y_i} \mbox{ and } \upsilon^j_y|_{\widetilde{A^j_y}}: \widetilde{A^j_y}
 \to A^j_{y_j}$$ are  biregular;
\item[(4)] $\mu \circ \upsilon^i_y = \sigma|_{\sV^i_y}$ and  $\mu
\circ \upsilon^j_y = \sigma|_{\sV^j_y},$ i.e., the following diagrams,
and those with $i$ and $j$ switched, commute.
$$
\begin{array}{ccccccccccc} & & \sV^i_y & \subset &
\widetilde{A^{ij}_x} & & & \;\;\;\;\;\; & \widetilde{A^i_y} & & \\
& & \downarrow \upsilon^i_y & & \downarrow \sigma & & & & \wr
\downarrow \upsilon^i_y & \searrow \sigma & \\ \sU & \supset &
A^{ij}_{y_i} & \stackrel{\mu}{\to} & A^{ij}_y & \subset & M & &
A^i_{y_i} & \stackrel{\mu}{\to} & A^i_y \end{array} $$
 \end{enumerate} \end{proposition}

\begin{proof}
Let $L^j_{y_i}$ be the connected component of $\mu^{-1}(L^j_y)$
through $y_i$. It is an open neighborhood of $y_i$ in $A^j_{y_i}$
and shrinking $U$ if necessary, $\rho|_{L^j_{y_i}}$ is a smooth
morphism from $y \in M^{\rm good}$ and the germ of the submanifold $\rho(L^j_{y_i}) \subset
\sK^{\rm bihol}$ at $\rho(y_i)$ is an irreducible component of the
germ of $K^j_{y_i}$ at $\rho(y_i)$.
Choose a neighborhood $O_{\rho(y_i)}$  as  in Definition \ref{d.web}
such that $\rho(L^j_{y_i}) \cap O_{\rho(y_i)}$ is smooth and
denote by $[\rho(L^j_{y_i}) \cap O_{\rho(y_i)}]$ its connected component through $\rho(y_i)$.
 Define $$\sV^i_{y_i} :=
\rho^{-1}([\rho(L^j_{y_i}) \cap O_{\rho(y_i)}]).$$    Then
$\mu|_{\sV^i_{y_i}}$ is an embedding into $M$. Its image lies
in $A^{ij}_y$ and contains an open subset of $A^{ij}_y$. Thus we
have a lifting $$\tilde{\mu}^i_{y_i}: \sV^i_{y_i} \to
\widetilde{A^{ij}_x}$$ which is an embedding satisfying $$ \mu|_{\sV^i_{y_i}} = \sigma
\circ \tilde{\mu}^i_{y_i} \mbox{ and }
\tilde{\mu}^i_{y_i}(A^i_{y_i}) = \widetilde{A^i_{y}}.$$ Define
$$\sV^i_y := \tilde{\mu}(\sV^i_{y_i}) \mbox{ and } \upsilon^i_y :=
(\tilde{\mu}^i_{y_i})^{-1}.$$  Define $\sV^j_y$ and
$\upsilon^j_y$ in the same way by switching $i$ and $j$.  Then all of
(1)-(4) are immediate from the construction. \end{proof}

\begin{proposition}\label{p.HilbA^ij}
In the setting of Proposition \ref{p.sigma}, the submanifolds
$\widetilde{A^i_x}$ and $\widetilde{A^j_x}$ of
$\widetilde{A^{ij}_x}$  belong to the same irreducible component
of the Hilbert scheme
 ${\rm Hilb}(\widetilde{A^{ij}_x})$ if and only if $A^{ij}_{x_i} = A^{ij}_{x_j}$.
\end{proposition}

\begin{proof}
Note that from Proposition \ref{p.sigma} (2), $\widetilde{A^i_x}$
and $\widetilde{A^j_x}$ correspond to smooth points of the Hilbert
scheme  ${\rm Hilb}(\widetilde{A^{ij}_x})$. So each of $\widetilde{A^i_x}$
and $\widetilde{A^j_x}$ belongs to a unique irreducible component
of the Hilbert scheme. At these smooth points, the Hilbert scheme
has dimension $$\dim A^{ij}_x - \dim A^i_x = \dim A^{ij}_x - \dim A^j_x = \rk(W^{ij}) -
\rk(W^i\cap W^j).$$   Thus, from Proposition \ref{p.sigma} (2) and (3),
all general deformations of $\widetilde{A^i_x}$ and $\widetilde{A^j_x}$ in
$\widetilde{A^{ij}_x}$ are proper transforms of the $\mu$-images of the fibers of the families
$$ A^{ij}_{x_i}\cap \rho^{-1}(\sK^{\rm bihol}) \to K^j_{x_i} \cap
\sK^{\rm bihol} \mbox{ and }  A^{ij}_{x_j}\cap \rho^{-1}(\sK^{\rm
bihol}) \to K^i_{x_j} \cap \sK^{\rm bihol},
$$  respectively. If $A^{ij}_{x_i} = A^{ij}_{x_j}$, then the two families
coincide. Thus $\widetilde{A^i_x}$ and $\widetilde{A^j_x}$ belong
to the same irreducible component of the Hilbert scheme
 ${\rm Hilb}(\widetilde{A^{ij}_x}).$

 On the other hand, if $A^{ij}_{x_i} \neq A^{ij}_{x_j}$,
from the effectiveness assumption in Definition \ref{d.web} (iii),
 deformations of $A^i_y$ and $A^j_y$ in $A^{ij}_x$  do not belong to
the same irreducible component of ${\rm Hilb}(A^{ij}_x)$. By
Proposition \ref{p.sigma} (2) and (3), this implies that
$\widetilde{A^i_x}$ and $\widetilde{A^j_x}$ do not belong to the same
irreducible component of
 ${\rm Hilb}(\widetilde{A^{ij}_x}).$
\end{proof}

\begin{proposition}\label{p.better}
In the setting of Proposition \ref{p.A^ij}, for a general point
$x\in M^{\rm good}$, the variety $A^{ij}_x$ is nonsingular at $x$
for any pair $(i,j) \in I \times I$.
\end{proposition}

\begin{proof}
From Proposition \ref{p.A^ij} (ii), it is clear that the
subvarieties in $M$ that can be realized as $A^{ij}_x$ form a family of dimension equal to $\dim
M - \dim A^{ij}_x$. Thus their singular loci are contained in a
proper subvariety in $M$. \end{proof}

\begin{proposition}\label{p.H_ij}
In the notation of Definition \ref{n.group} and Notation
\ref{n.pair}, choose $x \in M^{\rm good}$ general in the sense of
Proposition \ref{p.better}. Let ${\rm Sm}(A^{ij}_x)$ be the smooth
locus of $A^{ij}_x$ and let $H_{ij} \subset G$ be the image of
$$ \alpha \circ \lambda_{ij}: \pi_1({\rm Sm}(A^{ij}_x) \cap M^{\rm et}, x) \to \mathfrak{S}_X,$$
where the homomorphism $$\lambda_{ij}: \pi_1({\rm Sm}(A^{ij}_x)
\cap M^{\rm et}, x) \to \pi_1(M^{\rm et},x)$$ is induced by the
inclusion $A_x^{ij} \subset M$. Then $H_{ij}$ contains the
subgroups $H_i$ and $H_j$ in Definition \ref{n.group}.
\end{proposition}

\begin{proof}
The homomorphism
$$\lambda^o_i: \pi_1(A^i_x \cap {\rm Sm}(A^{ij}_x) \cap M^{\rm et}, x)  \to
\pi_1(M^{\rm et},x)$$ induced by the inclusion $A^i_x \subset M$
factors through $\lambda_{ij}$. Thus $H_{ij}$ contains the images
of $\alpha \circ \lambda^o_i$.

Denote by $\theta_i$  the homomorphism
$$ \theta_i: \pi_1(A^i_x \cap {\rm Sm}(A^{ij}_x) \cap M^{\rm et},
x) \to \pi_1(A^i_x \cap M^{\rm et}, x) $$ induced by the obvious
inclusion. Then $\lambda_i^o = \lambda_i \circ \theta_i$.  The
complement
$$(A^i_x \cap M^{\rm et}) \setminus (A^i_x \cap {\rm Sm}(A^{ij}_x)
\cap M^{\rm et}) $$ is a proper subvariety in the nonsingular
irreducible variety $A^i_x \cap M^{\rm et}$ from $A^i_x \subset
A_x^{ij}$ of Proposition \ref{p.leaf}. Thus $\theta_i$ must be
surjective.

It follows that $H_{ij}$ contains the image of $\alpha \circ
\lambda_i$  i.e., $H_i$. By the same reasoning, $H_{ij}$ contains
$H_j$, too.
\end{proof}

\begin{proposition}\label{p.separate}
In the setting of  Proposition
\ref{p.H_ij}, if $A^{ij}_{x_i} \neq A^{ij}_{x_j},$ then
$x_i$ and $x_j$ are not in the same
$H_{ij}$-orbit in $X = \mu^{-1}(x)$.
\end{proposition}

\begin{proof}
Since $A^{ij}_x$ is smooth at $x$, for each $x_k \in X$, there
exists a unique irreducible component of  $\mu^{-1}(A^{ij}_x)$
containing $x_k$. Thus the assumption $A^{ij}_{x_i} \neq
A^{ij}_{x_j}$ implies  $X \cap A^{ij}_{x_i} \neq X \cap
A^{ij}_{x_j}.$  To prove the proposition, it suffices to show that
$X \cap A^{ij}_{x_i}$ (resp. $ X \cap A^{ij}_{x_j}$) is the
$H_{ij}$-orbit of $x_i$ (resp. $x_j$).

 From Proposition \ref{p.leaf} (i), $A^{ij}_{x_i}$ is the
 irreducible component of $\mu^{-1}(A^{ij}_x)$ through $x_i$. Thus
$$A^{ij}_{x_i} \cap \mu^{-1}({\rm Sm}(A^{ij}_x) \cap M^{\rm et})$$
is precisely the connected component of $\mu^{-1}({\rm
Sm}(A^{ij}_x) \cap M^{\rm et})$ containing $x_i$.    Then  the
$H_{ij}$-orbit of $x_i$  is $$ X \cap A^{ij}_{x_i} \cap
\mu^{-1}({\rm Sm}(A^{ij}_x) \cap M^{\rm et}) = X \cap
A^{ij}_{x_i}$$ because $X = \mu^{-1}(x) \subset \mu^{-1}({\rm
Sm}(A^{ij}_x) \cap M^{\rm et})$ by our choice of $x$. In the same
way, the $H_{ij}$-orbit of $x_j$ is $X \cap A^{ij}_{x_j}$.
\end{proof}

\section{Pairwise integrable webs of tori}

The goal of this section is to prove Proposition \ref{p.toruslift}
about    pairwise integrable webs of tori on  projective
manifolds. Its proof requires some standard results on deformations
of submanifolds with trivial normal bundles. We start by
recalling them.

\begin{definition}\label{d.unobstruct}
A submanifold $A$ of a projective manifold $Z$ is {\em
unobstructed} if  the Hilbert scheme ${\rm
Hilb}(Z)$ of $Z$ is smooth at $[A] \in {\rm Hilb}(Z)$, the point
corresponding to $A$. In this case, denote by ${\rm Hilb}(Z)_A$
the (unique) irreducible component of ${\rm Hilb}(Z)$ containing
$[A]$, by $\xi_A: {\rm Univ}(Z)_A \to {\rm Hilb}(Z)_A$ the
universal family and by $\eta_A: {\rm Univ}(Z)_A \to Z$ the
evaluation morphism.
\end{definition}

\begin{proposition}\label{p.unobstruct}
For an unobstructed submanifold $A$ with trivial normal
bundle in a projective manifold $Z$, denote by ${\rm Hilb}(Z)_A^o
\subset {\rm Hilb}(Z)_A$  the open subset consisting of points
parametrizing unobstructed submanifolds with trivial normal
bundles. Denote by $\xi_A^o: {\rm Univ}(Z)_A^o \to {\rm
Hilb}(Z)_A^o$ and  $\eta_A^o: {\rm Univ}(Z)_A^o \to Z$, the
restrictions of $\xi_A$ and $\eta_A$, respectively.  Then
$\xi_A^o$ is a smooth projective morphism and  $\eta_A^o$ is
unramified. If furthermore $\eta_A$ is birational, then $\eta_A^o$
is biregular to its image which is a Zariski open subset in $Z$.
\end{proposition}

\begin{proof}
Since all members of ${\rm Hilb}(Z)_A^o$ are smooth, $\xi_A^o$ is
a smooth projective morphism.
 The evaluation morphism $\eta^o_A$ is unramified by the
triviality of the normal bundles of members of ${\rm
Hilb}(Z)_A^o$. If $\eta_A$ is birational, the unramified morphism
$\eta_A^o$ must be biregular over its image.
\end{proof}

\begin{remark} Proposition \ref{p.unobstruct} implies that an unobstructed
submanifold with trivial normal bundle is a member of a web of submanifolds. Conversely, a member of a web of submanifolds is
an unobstructed submanifold with trivial normal bundle. Thus one can replace  Definition \ref{d.web} by Proposition \ref{p.unobstruct}
and develop all the theory starting from there. But we prefer Definition \ref{d.web}, because it is more geometrically appealing (to us) and also the approach via Hilbert scheme  plays a rather restricted role
in this paper; it is used essentially  only in this section. \end{remark}

%

%
%
%

\begin{proposition}\label{p.subtorus}
Assume that $A$ in Proposition \ref{p.unobstruct} is biregular to
a complex torus. Suppose there exists a subtorus $S \subset A$
such that deformations of $S$ in $Z$ cover an open subset in $Z$. Then we
have the following.
\begin{enumerate} \item[(i)] There exists an open neighborhood $U
\subset Z$ of $A$ equipped with a smooth projective morphism $f: U
\to \Delta^n$ over a polydisc $\Delta^n$ of dimension $n = \dim Z
- \dim A$,  such that fibers of $f$ give deformations of $A$ in
$U$. \item[(ii)]  There exists a smooth projective morphism
$\zeta: U \to U'$ over a complex manifold $U'$ whose fibers are
deformations of $S$. Thus $S$ is unobstructed with trivial normal bundle in $Z$ and $\zeta$
induces a natural embedding of $U'$ into ${\rm
Hilb}(Z)_S^o$, realizing $U'$ as an open neighborhood of  $[S] \in
{\rm Hilb}(Z)_S^o$. \item[(iii)] There exists a smooth projective
morphism $f': U' \to \Delta^n$ with $f= f' \circ \zeta$ such that
the fibers of $f'$ are the quotient tori of  deformations of $A$
by  deformations of $S$.
\end{enumerate}
\end{proposition}

\begin{proof}
By Proposition \ref{p.unobstruct}, we can choose a polydisc
neighborhood $\Delta^n \subset {\rm Hilb}(Z)_A^o$ of $[A]$ such that
setting  $U^1:= \xi_A^{-1}(\Delta^n)$, the morphism
$\eta_A|_{U^1}$ is biholomorphic to its image $U:= \eta_A(U^1) \subset Z$. Then
we have $f: U \to \Delta^n$ satisfying $$\xi_A|_{U^1} = f \circ
\eta_A|_{U^1}.$$ Since $\Delta^n$ contains no positive-dimensional
compact subvariety, all deformations of $S$ in $U$ are contained
in fibers of $f$. By shrinking $U$, we can assume that there
exists a section $\Sigma \subset U$ of $f$ which intersects $S$
and is contained in the locus of deformations of $S$. This makes
$f:U \to \Delta^n$ into a holomorphic family of complex torus
groups (analytic abelian scheme over $\Delta^n$) equipped with a
family of subtori.
Let $f': U' \to
\Delta^n$ be the family of quotient groups with $\zeta: U \to U'$
the quotient morphism. Then $U', f'$ and $\zeta$ have the
required properties.
\end{proof}

\begin{proposition}\label{p.zeta}
In the setting of Proposition \ref{p.subtorus}, assume that
$\eta_S: {\rm Univ}(Z)_S \to Z$ is birational. Then we have the
following. \begin{itemize} \item[(a)] There exists a Zariski open
subset $Z^o \subset Z$ with a smooth projective morphism $\zeta:
Z^o \to \sM^o$ to a smooth variety $\sM^o$ whose fibers  are
deformations of $S$. \item[(b)]
 For a general member  $[A^t] \in {\rm Hilb}(Z)_A^o$, there exists $[S^t] \in {\rm Hilb}(Z)_S^o$
 such that $S^t \subset A^t$ is a subtorus. \item[(c)] $ A^t$ in (b) lies in $Z^o$ and $\zeta(A^t) \subset
  \sM^o$
 is an unobstructed torus with trivial normal bundle, biregular to
  the quotient torus $A^t/S^t$.
  \end{itemize}
 \end{proposition}

 \begin{proof}
 By Proposition \ref{p.unobstruct} applied to $S$ in place of $A$,
we obtain a Zariski open subset $Z^o \subset Z$ as the image of
$\eta_S^o$ and a morphism $\zeta: Z^o \to \sM^o$ over a smooth
variety $\sM^o \cong {\rm Hilb}(Z)_S^o$, proving (a).

Since deformations of  $S$ cover an open subset in $Z$, Proposition
\ref{p.subtorus} shows that a general deformation of $A$ contains a
deformation of $S$ as a subtorus. This proves (b).

To see (c), apply
  Proposition \ref{p.subtorus}  in a neighborhood of $A^t$ containing $S^t$.
  All translates  of $S^t$ inside $A^t$
 are elements of ${\rm Hilb}(Z)^o_S$, implying $A^t \subset Z^o$.
 That $\zeta (A^t)$ is unobstructed and has trivial normal bundle
 is immediate from   Proposition \ref{p.subtorus} (iii).
 \end{proof}

\begin{proposition}\label{p.intersection}
Let $A$ be an unobstructed submanifold in a projective manifold
$Z$ with trivial normal bundle such that $\dim Z = 2 \dim A$. For
any $[A'] \subset {\rm Hilb}(Z)_A^o$,  the intersection $A \cap
A'$ has no isolated point.
\end{proposition}

\begin{proof}
From $\dim Z = 2 \dim A$,  the intersection number $A \cdot A'$ is
well-defined and equal to $A \cdot A$.  Since a small deformation of
$A$ is disjoint from $A$ by the triviality of the normal bundle,
we have $A \cdot A' =0.$

Suppose that $A \cap A'$ has an  isolated point $z$. Regard $A
\cdot A'$ as an intersection cycle in the sense of \cite{Fu}. The
isolated intersection point $z$ gives a positive contribution to
$A \cdot A'$. The contribution from the other components of $A
\cap A'$ is non-negative by Theorem 12.2 of \cite{Fu} because the
normal bundle of $A$ is trivial. This gives $A \cdot A'
>0$, a contradiction.
\end{proof}

\begin{proposition}\label{p.neq}
Let $A^1, A^2$ be two unobstructed tori with trivial
normal bundles in a projective manifold $Z$. Assume that a connected component $S$ of
$A^1 \cap A^2$ is a subtorus both in $A^1$ and in $A^2$, with
$\dim Z = \dim A^1 + \dim A^2 - \dim S$. Assume furthermore that
$S$ is unobstructed with trivial normal bundle in $Z$ and $\eta_S:
{\rm Univ}(Z)_S \to Z$ is birational.   Then ${\rm Hilb}(Z)_{A^1}
\neq {\rm Hilb}(Z)_{A^2}$.
\end{proposition}

\begin{proof} Applying Proposition \ref{p.zeta} to $S \subset Z$ with $A =A^1$ (resp. $A=A^2$),
 we see that
$\zeta(A^1)$ (resp. $\zeta(A^2)$) is an unobstructed torus with trivial
normal bundle in a projective manifold $\sM$ compactifying
$\sM^o$. Since $S$ is a component of $A^1\cap A^2$, we see that
$\zeta(A^1) \cap \zeta(A^2)$ has an isolated point $\zeta(S)$. If
${\rm Hilb}(Z)_{A^1} = {\rm Hilb}(Z)_{A^2}$, then $\dim A^1 = \dim
A^2$ and $\zeta(A^2)$ is a member of ${\rm
Hilb}(\sM)_{\zeta(A^1)}$. From the assumption $\dim Z = \dim A^1 +
\dim A^2 - \dim S$, we have
$$\dim \sM = \dim Z - \dim S = 2 (\dim A^1 - \dim S) = 2\dim
\zeta(A^1) = 2\dim \zeta(A^2).$$ Applying Proposition
\ref{p.intersection} with $A:= \zeta(A^1)$ and $A' :=
\zeta(A^2)$, we have a contradiction.
\end{proof}

We will skip the proof of the following  elementary lemma.

\begin{lemma}\label{l.subtorus} A closed submanifold of a complex torus with trivial normal
bundle is a subtorus. \end{lemma}

\begin{proposition}\label{p.weboftori}
Let $Z$ be a projective manifold and let $A_1, A_2 \subset
Z$ be two distinct tori with  $A_1 \cap A_2 \neq \emptyset$. Assume that there
exist open neighborhoods $A_1 \subset \sV_1$ and $A_2 \subset
\sV_2$ equipped with  smooth projective morphisms $$\rho_1: \sV_1
\to \Delta^{\dim Z - \dim A^1} \mbox{ and } \rho_2: \sV_2 \to
\Delta^{\dim Z - \dim A^2}$$ such that $ A_1= \rho_1^{-1}(0)$ and $A_2=
\rho_2^{-1}(0).$ Then for a general point $u \in \sV_1 \cap
\sV_2$, the connected component $S_u$ of $$\rho_1^{-1}(\rho_1(u))
\cap \rho_2^{-1}(\rho_2(u))$$ with $u \in S_u$ is a subtorus both in
$\rho_1^{-1}(\rho_1(u))$ and in $\rho_2^{-1}(\rho_2(u)).$
Furthermore, this $S_u$ is unobstructed with trivial normal bundle
in $Z$.
\end{proposition}

\begin{proof}
 For each $t \in \Delta^{\dim Z - \dim A^2}$, $\rho_2^{-1}(t) \cap \sV_1$ is a closed subvariety of $\sV_1$. General fibers of
the proper morphism $\rho_1|_{\rho_2^{-1}(t) \cap \sV_1}$ are  unions of submanifolds with
trivial normal bundles in the torus $\rho_2^{-1}(t)$. Thus they
must be  subtori of $\rho_2^{-1}(t)$  by Lemma~\ref{l.subtorus}.
It follows that for general $u \in \sV_1 \cap \sV_2$, $S_u$ must be a
subtorus of $\rho_2^{-1}(\rho_2(u))$, and by the same reasoning, also a
subtorus in $\rho_1^{-1}(\rho_1(u))$. By the generality of $u$,
deformations of $S_u$ cover an open subset in $Z$. Applying Proposition
\ref{p.subtorus} with $S= S_u$ and $A= \rho_1^{-1}(\rho_1(u))$,   we see that $S$ is
unobstructed with trivial normal bundle in $Z$.
\end{proof}

The next proposition is the main result of this section.

\begin{proposition}\label{p.toruslift}
Let  $\sW$ be a pairwise integrable web on a projective manifold
$M$ whose members are tori. Fix a general point $x\in M^{\rm
good}$ and choose a neighborhood $U \subset M^{\rm good}$ as in
Definition \ref{d.good}.  Since $A^{ij}_x$ is smooth at $x$ by
Proposition \ref{p.better}, we may assume  by shrinking $U$ that
 $$\sigma|_{\sigma^{-1}(U)}: \sigma^{-1}(U) \to U \cap  A^{ij}_x$$
 is biholomorphic.  Using the notation of
Proposition \ref{p.sigma} and shrinking $U$ further if necessary, we have
the following.
\begin{enumerate} \item[(i)] For any $y \in U \cap A^{ij}_x$, the component $S^{ij}_y$ of
$\widetilde{A^i_y} \cap \widetilde{A^j_y}$ through
$\sigma^{-1}(y)$ is unobstructed with trivial normal bundle in
$\widetilde{A^{ij}_x}$ and is  a subtorus both in
$\widetilde{A^i_x}$ and $\widetilde{A^j_x}$. \item[(ii)] The germ of $S^{ij}_y$ at $y$ is sent by
$\sigma$ to that of the
leaf of $W^i\cap W^j$ through $y$. \item[(iii)] Small deformations
of $S^{ij}_y$ in $\widetilde{A^{ij}_x}$ are realized by $S^{ij}_z$
for $z \in U \cap A^{ij}_x$.   \item[(iv)] If $\eta_{S^{ij}_x}:
{\rm Univ}(\widetilde{A^{ij}_x})_{S^{ij}_x} \to
\widetilde{A^{ij}_x}$ is birational, then $A^{ij}_{x_i} \neq
A^{ij}_{x_j}.$
\end{enumerate}
\end{proposition}

\begin{proof}
(i) is a direct consequences of Proposition \ref{p.sigma} and
Proposition \ref{p.weboftori}. (ii) is immediate from Proposition
\ref{p.A^ij} (i).

For (iii), note that for $z \in U \cap A^{ij}_x$ with sufficiently small $U$, the  submanifold
 $S^{ij}_z$ is a deformation of $S^{ij}_y$ and the family $\{
S^{ij}_z, \; z \in U \cap A^{ij}_x\}$
covers an open set in $\sigma^{-1}(U).$ Since $S^{ij}_y$ has trivial normal
bundle in $\widetilde{A^{ij}_x}$, this implies that the family $\{
S^{ij}_z, \; z \in U \cap A^{ij}_x\}$  corresponds to  an open neighborhood
of $[S^{ij}_y]$ in ${\rm Hilb}(\widetilde{A^{ij}_x})$, proving (iii).

To prove (iv), set $$Z=\widetilde{A^{ij}_x}, \; A^1:=
\widetilde{A^i_x} \mbox{ and }A^2:= \widetilde{A^j_x}.$$ From (ii)
we have $\dim Z = \dim A^1 + \dim A^2 - \dim S$.  Applying
Proposition \ref{p.neq}, we have ${\rm Hilb}(Z)_{A^1} \neq {\rm
Hilb}(Z)_{A^2}$. Thus $A^{ij}_{x_i} \neq A^{ij}_{x_j}$ by
Proposition \ref{p.HilbA^ij}. \end{proof}

\section{Webs of Lagrangian tori}\label{s.Lagrangian}

\begin{definition}\label{d.symplecticVector}
An even-dimensional vector space $V$ equipped with a
non-degenerate form $\omega \in \wedge^2 V^*$ is a {\em symplectic
vector space}. Given a subspace $W \subset V$, define $$W^{\perp}
:= \{ v \in V, \omega(v, w) = 0 \mbox{ for all } w \in W\}.$$  A
subspace $W \subset V$ is {\em Lagrangian} if $W = W^{\perp}$. If $W$ is Lagrangian, $2 \dim W = \dim V$.
\end{definition}

We will skip the proof of the following elementary lemma.

\begin{lemma}\label{l.Lagrangian}
Given a symplectic vector space $(V, \omega)$ and two Lagrangian
subspaces $W^1, W^2 \subset V$,
$$W^1 \cap W^2 = (W^1 + W^2)^{\perp} \subset W^1 + W^2.$$
\end{lemma}

\begin{definition}\label{d.symplectic}
Let $M$ be a complex manifold. A {\em symplectic form} on $M$ is a
closed holomorphic 2-form $\omega$ on $M$ such that for each $x\in
M$, the tangent space $(T_x(M), \omega_x)$ is a symplectic vector
space. The pair $(M, \omega)$ is called a {\em symplectic
manifold}. A submanifold $A$ of a symplectic manifold $(M,
\omega)$ is {\em Lagrangian} if for each $x \in A$, $T_x(A)$ is a
Lagrangian subspace of $T_x(M)$. A Lagrangian submanifold is a
{\em Lagrangian torus} if it is biholomorphic to a complex torus.
A flat morphism with connected fibers $f: M \to B$ from a
symplectic manifold $M$ onto a complex manifold $B$ is a {\em
Lagrangian fibration} if each general fiber is a Lagrangian
submanifold.
\end{definition}

\medskip
The next lemma is standard: see, e.g., p.220 of \cite{GS}.

\begin{lemma}\label{l.GS}
The cotangent bundle $T^*(B)$ of a complex manifold $B$ has a
canonical holomorphic symplectic form $\omega_{\rm st}$ on it. When $n=\dim B,$ a holomorphic coordinate system $(q^1, \ldots, q^n)$ on a
coordinate neighborhood $U \subset B$ and the associated functions
$p^i$ on $T^*(U)$ given by $\frac{\partial}{\partial q^i}$ define a
holomorphic coordinate system $(p^1, \ldots, p^n, q^1, \ldots, q^n)$ on $T^*(U)$ such that
$$\omega_{\rm st}|_{T^*(U)} = dp^1 \wedge dq^1 + \cdots + dp^n \wedge dq^n.$$
With respect to $\omega_{\rm st}$, the natural projection $\pi:
T^*(B) \to B$ is a Lagrangian fibration and a section $\Sigma
\subset T^*(B)$ of $\pi$ is a Lagrangian submanifold if and only
if it is $d$-closed when regarded as a 1-form on $B$.
\end{lemma}

The following is a holomorphic version of the action-angle
variables for completely integrable Hamiltonian systems, Theorem
44.2 in \cite{GS}. It is a reformulation of Proposition 3.5 in
\cite{Hw}, and its proof will be skipped.

\medskip
\begin{proposition}\label{p.GS}  Let $(N, \omega)$ be a symplectic manifold
and $f: (N, \omega) \rightarrow B$ be a proper Lagrangian
fibration such that each fiber is a complex torus and there exists
a Lagrangian section $\Sigma \subset N$. Then there exists an
unramified surjective holomorphic map $\chi: T^*(B) \rightarrow N$
with $\pi= f \circ \chi$ such that \begin{enumerate} \item[(i)] for each $b \in B$,
$\chi_b: T^*_b(B) \to f^{-1}(b)$ is the universal covering of the
complex torus with $\chi_b(0) = \Sigma \cap f^{-1}(b)$, \item[(ii)] in
the notation of Lemma \ref{l.GS},  $\omega_{\rm st} = \chi^*
\omega,$ and, consequently, \item[(iii)] each component of
$\chi^{-1}(\Sigma)$ is a Lagrangian submanifold in $T^*(B)$,
locally defining a closed 1-form on $B$. \end{enumerate}
\end{proposition}

\begin{proposition}\label{p.integrable}
In the setting of Proposition \ref{p.GS},  let $O \subset N$ be a
connected open subset equipped with a smooth  Lagrangian fibration
$\psi: O \to Q$, different from $f|_U$. For each $x \in U,$
consider the subspace $\sD_x \subset T_x(U)$ defined by
$$\sD_x :=  T_x(f^{-1}(f(x)))+ T_x(\psi^{-1}(\psi(x))).$$
By shrinking $O$ if necessary, we may assume that $\{\sD_x, x \in O\} $ defines
 a vector subbundle $\sD \subset T(O)$. If for each $x \in O$, $f^{-1}(f(x)) \cap
\psi^{-1}(\psi(x))$ is an open subset in a subtorus of
$f^{-1}(f(x))$, then $\sD$ is integrable.
\end{proposition}

\begin{proof}
At each point $x \in O$, define $F_x \subset T_x(O)$ by
$$ F_x:= T_x(f^{-1}(f(x))) \cap T_x(\psi^{-1}(\psi(x))).$$
By shrinking $O$ if necessary,  $F \subset T(O)$
defines a foliation on $O$, whose leaves are open subsets of subtori in the fibers of $f$. After
shrinking $B$,  if necessary, we can extend $F$ via translations in the fibers of $f$ to a
foliation on $N$ whose leaves are subtori of the fibers of $f$. Let $\sS \subset N$ be the collection of these subtori intersecting $\Sigma$ such that $f|_{\sS}: \sS \to B$ is a family of subtori with a section $\Sigma \subset \sS$ and leaves of $F$ are just translates of fibers of $f|_{\sS}$ (as in
Proposition \ref{p.subtorus}).

By
Proposition \ref{p.GS} and shrinking $B$ further, we have  $\chi: T^*(B) \to N$ such that
$\chi^{-1}(\Sigma)$ is a union of Lagrangian sections of $T^*(B)
\to B$.  The component of $\chi^{-1}(\sS)$
containing the zero-section of $T^*(B)$ is a  vector subbundle
$\sF \subset T^*(B)$ such that $$\chi|_{\sF_b}: \sF_b \to \sS_b:= f^{-1}(b) \cap \sS$$  is the universal cover of the subtorus $\sS_b$ for each $b \in B$.

 Set $n= \dim B$ and $r= \rk(F)$.
We can find a set  $\{ \Sigma^1, \ldots, \Sigma^n \}$ of components of  $\chi^{-1}(\Sigma)$ forming a frame for the vector bundle $T^*(B)$
such that the subset $\{\Sigma^1, \ldots, \Sigma^r\}$ forms a frame for the subbundle $\sF \subset T^*(B)$.  Using Lemma \ref{l.GS} and  shrinking $B$ further, we have holomorphic functions $q^1, \ldots, q^n$ on $B$ such that
 the  closed 1-forms
$dq^1, \ldots, dq^n$ represent $\Sigma^1, \ldots, \Sigma^n \subset T^*(B)$. Then in terms of the  coordinates $$(p^1=
\frac{\partial}{\partial q^1}, \ldots, p^n
=\frac{\partial}{\partial q^n}, q^1, \ldots, q^n)$$ on $T^*(B)$ of
Lemma \ref{l.GS},  the foliation $\chi^{-1}F$  induced
by $F$ on $T^*(B)$ is given by the translates of the span of $\frac{\partial}{\partial p^1},\;
\ldots, \; \frac{\partial}{\partial p^r}$.  Since the distribution
$\sD$ coincides with $F^{\perp}$ by Lemma \ref{l.Lagrangian}, the
pull-back distribution $\chi^{-1} \sD$ on $T^*(B)$ is given by the translates of the span of
$$\frac{\partial}{\partial p^1}, \; \ldots, \; \frac{\partial}{\partial
p^n}, \; \frac{\partial}{\partial q^{n-r}}, \; \ldots,
\; \frac{\partial}{\partial q^n}.$$ Thus $\chi^{-1} \sD$ is
integrable and so is $\sD$.
\end{proof}

\begin{definition}\label{d.webtori}
Let $M$ be a projective symplectic manifold, i.e., a projective
manifold equipped with a holomorphic symplectic form. A web of
submanifold $\sW$ on $M$ is called a {\em web of Lagrangian tori}
if its members are Lagrangian tori in $M$.
\end{definition}

\begin{proposition}\label{p.integrable2}
A web of Lagrangian tori is pairwise integrable.
\end{proposition}

\begin{proof}
Let $\sW=[\mu: \sU \to M, \rho: \sU \to \sK]$ be a web of
Lagrangian tori on a projective symplectic manifold $(M, \omega)$.
The open set $N:= \rho^{-1}(\sK^{\rm bihol})$ is a symplectic
manifold equipped with the symplectic form $(\mu|_N)^* \omega$ and $
\rho|_N: N \to  \sK^{\rm bihol}$ is a Lagrangian fibration. We
choose $U \subset M^{\rm good}$ and   $U^i \subset N$
as in Definition \ref{d.good}. Let $\rho^i:= \rho|_{U^i}$. By the
natural biholomorphic map $ (\mu^1)^{-1} \circ \mu^i: U^i \to
U^1$, we can regard $\rho^i$ as defined on $U^1$. Then $$f= \rho|_N,
\; \psi:= \rho^i, \; B := \sK^{\rm bihol}, \; \sD := W^{1i} \mbox{
and } O:= U^1$$ satisfy the assumption of Proposition
\ref{p.integrable} by Proposition \ref{p.toruslift} (i). Thus
Proposition~\ref{p.integrable} implies that $W^{1i}$ is
integrable. By the same reasoning, we get the integrability of
$W^{ij}$ for all pairs $(i,j)$. Thus $\sW$ is pairwise integrable.
\end{proof}

\begin{proposition}\label{p.univalent}
In the setting of Proposition \ref{p.integrable2}, use the
notation of Proposition  \ref{p.toruslift}. For a general point $x
\in M^{\rm good}$, we have $Z := \widetilde{A^{ij}_x}$ and $S:=
S^{ij}_x\subset Z$, an unobstructed torus with trivial normal
bundle in $Z$ which is the connected component of
$\widetilde{A^i_x} \cap \widetilde{A^j_x}$ containing
$\sigma^{-1}(x)$.  Then $\eta_S: {\rm Univ}(Z)_S \to Z$ is
birational.
\end{proposition}

\begin{proof}
 We claim that there exists a non-empty Zariski open
subset $\sZ \subset Z$ and a vector subbundle $\sN \subset T(\sZ)$
such that small deformations of $S$ in $Z$ intersecting $\sZ$  are tangent to $\sN$ with
$\rk(\sN) = \dim S$. This implies that $\eta_S$ is birational.

To prove the claim, define the null subspace at a smooth point $y
\in A^{ij}_x$ by
$${\rm Null}^{ij}(\omega)_y := \{ v \in T_y(A^{ij}_x),
\omega(v, T_y(A^{ij}_x)) =0\}= T_y(A^{ij}_x)^{\perp} \cap
T_y(A^{ij}_x).$$ On a Zariski open subset $\sA \subset {\rm
Sm}(A^{ij}_x)$, this defines a vector subbundle  ${\rm
Null}^{ij}(\omega)|_{\sA} \subset T(\sA)$. Using the desingularization
$\sigma: Z \to A^{ij}_x$, define $$\sZ := \sigma^{-1}(\sA) \cong
\sA \mbox{ and } \sN := d \sigma^{-1}({\rm Null}^{ij}(\omega)|_{\sA})\;
\subset T(\sZ).$$  When $\{W^i, i \in I\}$ is the regular web
induced by $\sW$ in a neighborhood $U$ of $x$, Lemma~\ref{l.Lagrangian} gives, for a point $y \in \sA \cap U$,
$$W^i_y \cap W^j_y  = (W^{ij}_y)^{\perp} \cap W^{ij}_y = T_y(A^{ij}_x)^{\perp} \cap T_y(A^{ij}_x)=
 {\rm Null}^{ij}(\omega)_y$$ because the germ of $\sA$ is that of a leaf of $W^{ij}$ by Proposition
 \ref{p.A^ij} (ii).  Since
the germs of deformations of $S$ in $Z$ correspond to those of
leaves of $W^i \cap W^j|_{\sA}$ by Proposition \ref{p.toruslift}
(ii) and (iii), we see that deformations of $S$ are tangent to
$\sN$ with $\rk(\sN) = \dim S$, proving the claim.
\end{proof}

Now we are ready to complete the proof of Theorem \ref{t.main}.

\begin{proof}[Proof of Theorem \ref{t.main}]
It is well-known that a Lagrangian torus $A \subset M$ is
unobstructed with trivial normal bundle (e.g. by Theorem 8.7 in
\cite{DM}). Thus we have $\xi_A^o: {\rm Univ}(M)_A^o \to {\rm
Hilb}(M)_A^o$ and $\eta_A^o: {\rm Univ}(M)_A^o \to M$ in
Proposition \ref{p.unobstruct}. By choosing suitable projective
manifolds compactifying ${\rm Univ}(M)_A^o$ and ${\rm
Hilb}(M)_A^o$, we obtain a
 web of Lagrangian tori $\sW=[\mu: \sU \to M, \rho: \sU \to \sK]$ which has $A$ as a member.

Using the notation of Definition \ref{n.group}, suppose that the
group $\sH \subset G$ generated by $H_1, \ldots, H_d$ acts
intransitively on $X$. Then by Proposition \ref{p.equivalence}, we
have a factorization of $\mu: \sU \to M$ via a generically finite morphism $\mu': \sU' \to M$
which is not birational and splits over a general member of $\sW$.
Since $M$ is simply connected, the branch divisor $D \subset M$ of
$\mu'$ is a non-empty hypersurface. By Proposition \ref{p.split},
$D$ is disjoint from a general member of $\sW$ and we are done.

Thus we may assume that $\sH$ acts transitively on $X$. We claim
that $(X, G, H)$ with $H=H_1$ is a special triple in the sense of
Definition \ref{d.triple}. From Proposition \ref{p.conjugate}, $H$
is a normal subgroup of the isotropy subgroup $G_1$ of $x_1 \in X$
and $H_i= H_{x_i} = g_i H g_i^{-1}$ when $x_i = g_i \cdot x_1.$ In
terms of Definition \ref{d.triple}, $\langle X \rangle = \sH$.
Thus our assumption that $\sH$ acts transitively on $X$ is exactly
Definition \ref{d.triple} (1).
 The web
$\sW$ is pairwise integrable by Proposition \ref{p.integrable2}.
For a general point $x$ and any pair $x_i \neq x_j$ of points on
$X= \mu^{-1}(x)$, we have the torus $S$ in
$Z=\widetilde{A^{ij}_x}$ with $\eta_S$ birational by Proposition
\ref{p.univalent}. This implies $A^{ij}_{x_i} \neq A^{ij}_{x_j}$   by Proposition \ref{p.toruslift}. Applying Proposition \ref{p.separate},
we see that $x_i \neq x_j$ do not belong to the same $H_{ij}$-orbit in
$X$. Since $H_i, H_j \subset H_{ij}$ by Proposition \ref{p.H_ij},
 $x_i$ and $x_j$ do not belong to the same $\langle x_i, x_j \rangle$-orbit. This is
the condition (2) of Definition \ref{d.triple}. Thus the triple
$(X, G, H_1)$ is a special triple.

By Theorem \ref{t.triple}, we see that $d=1$, i.e., $\mu$ is
birational. By taking a general ample hypersurface $D' \subset
\sK$ and letting $D = \mu(\rho^{-1}(D'))$, we see that
 $A$ is disjoint from $D$.

\end{proof}

\end{document}